\documentclass[12pt]{amsart}

\usepackage{amsmath}
\usepackage{amssymb}
\usepackage{amsthm}
\usepackage{comment}

\newcommand{\splitTX}{\ensuremath{\mathcal{O}_{\mathbb{P}^1}(2) \oplus \mathcal{O}_{\mathbb{P}^1}(1)^{\oplus d} \oplus \mathcal{O}_{\mathbb{P}^1}^{\oplus (n - d - 1)}}}

\newtheorem{conjecture}{Conjecture}[section]
\newtheorem{lemma}[conjecture]{Lemma}
\newtheorem{theorem}[conjecture]{Theorem}
\newtheorem{claim}{Claim}[conjecture]
\newtheorem{corollary}[conjecture]{Corollary}
\newtheorem{proposition}[conjecture]{Proposition}

\theoremstyle{definition}
\newtheorem{remark}[conjecture]{Remark}
\newtheorem{example}[conjecture]{Example}

\numberwithin{equation}{conjecture}

\newenvironment{enumerate-p}{
  \begin{enumerate}}
  {\global\setcounter{equation}{\value{enumi}}\end{enumerate}}
\newenvironment{enumerate-cont}{
  \begin{enumerate}
    {\setcounter{enumi}{\value{equation}}}}
  {\setcounter{equation}{\value{enumi}}
  \end{enumerate}}
\DeclareMathOperator{\degree}{deg}
\DeclareMathOperator{\RatCurves}{RatCurves}

\begin{document}

\title[Characterizations of Projective Spaces and Hyperquadrics]{Characterizations of Projective Spaces and Hyperquadrics via Positivity Properties of the Tangent Bundle} 
\author{Kiana Ross}
\address{Department of Mathematics, University of Washington, Seattle, WA USA}
\date{\today}
\maketitle

\begin{abstract}
Let $X$ be a smooth complex projective variety. A recent conjecture of S. Kov\'acs states that if the $p^{\text{th}}$-exterior power of the tangent bundle $T_X$ contains the $p^{\text{th}}$-exterior power of an ample vector bundle, then $X$ is either a projective space or a smooth quadric hypersurface. This conjecture is appealing since it is a common generalization of Mori's, Wahl's, Andreatta-W\'isniewski's, Kobayashi-Ochiai's and Araujo-Druel-Kov\'acs's characterizations of these spaces. In this paper I give a proof affirming this conjecture for varieties with Picard number 1. 
\end{abstract}

\section{Introduction}\label{S:intro}

Let $X$ be a smooth complex projective variety of dimension $n$. In a seminal paper \cite{MR554387}, S. Mori proved that the only such varieties having ample tangent bundle $T_X$ are projective spaces. This result finally settled Hartshorne's conjecture \cite{MR0282977}, the algebraic analog of Frankel's conjecture \cite{MR0123272} in complex differential geometry. (Another proof of Frankel's conjecture was given around the same time by Y. Siu and S. Yau in \cite{MR577360} using harmonic maps.) Since then, the ideas of \cite{MR554387} have been expanded significantly, and there are many results in the literature using positivity properties of $T_X$ to characterize projective spaces and quadric hypersurfaces. In this paper I will prove another characterization in this direction:

\begin{theorem}\label{T:main}
Let $X$ be a smooth complex projective variety of dimension $n$ with Picard number $1$. Assume that there exists an ample vector bundle $\mathcal{E}$ of rank $r$ on $X$ and a positive integer $p \leq r$ such that $\wedge^p \mathcal{E} \subseteq \wedge^p T_X$. Then either $X \simeq \mathbb{P}^n$, or $p = n$ and $X \simeq Q_p \subset \mathbb{P}^{p+1}$, where $Q_p$ denotes a smooth quadric hypersurface in $\mathbb{P}^{p+1}$.
\end{theorem}

Theorem~\ref{T:main} gives an affirmative answer for varieties with Picard number 1 of the following more general conjecture of S. Kov\'acs:

\begin{conjecture}[Kov\'acs]\label{Q:main}
Let $X$ be a smooth complex projective variety of dimension $n$. If there exists an ample vector bundle $\mathcal{E}$ of rank $r$ on $X$ and a positive integer $p \leq r$ such that $\wedge^p \mathcal{E} \subseteq \wedge^p T_X$, then either $X \simeq \mathbb{P}^n$, $p = n$ and $X \simeq Q_p \subset \mathbb{P}^{p+1}$, where $Q_p$ denotes a smooth quadric hypersurface in $\mathbb{P}^{p+1}$.
\end{conjecture}

Motivation for this conjecture comes from the desire to unify existing characterization results of this type into a single statement. Mori's proof of the Hartshorne conjecture in 1979 was the first major result, and its method of studying rational curves of minimal degree has been a catalyst for much that has followed.

\begin{theorem}\label{T:Mori} \cite{MR554387}
Let $X$ be a smooth complex projective variety of dimension $n$, and assume that the tangent sheaf $T_X$ is ample. Then $X \simeq \mathbb{P}^n$.
\end{theorem}

In 1983, J. Wahl proved a related statement using algebraic methods: 

\begin{theorem}\label{T:Wahl} \cite{MR700774}
Let $X$ be a smooth complex projective variety of dimension $n$, and assume that the tangent sheaf $T_X$ contains an ample line bundle $\mathcal{L}$. Then either $(X, \mathcal{L}) \simeq (\mathbb{P}^n, \mathcal{O}_{\mathbb{P}^n}(1))$ or $(X, \mathcal{L}) \simeq (\mathbb{P}^1, \mathcal{O}_{\mathbb{P}^1}(2))$.
\end{theorem}

Note that S. Druel gave a geometric proof of this theorem in \cite{MR2092774}. In 1998, F. Campana and T. Peternell generalized Wahl's theorem to bundles of rank $r = n, n - 1$, and $n - 2$ \cite{MR1642626}. Finally, in 2001, M. Andreatta and J. Wi\'sniewski proved the analogous statement for vector bundles of arbitrary rank:

\begin{theorem}\label{T:AW01} \cite{MR1859022}
Let $X$ be a smooth complex projective variety of dimension $n$, and assume that the tangent sheaf $T_X$ contains an ample vector bundle $\mathcal{E}$ of rank $r$. Then either $(X, \mathcal{E}) \simeq (\mathbb{P}^n, \mathcal{O}_{\mathbb{P}^1}(1)^{\oplus r})$ or $r = n$ and $(X, \mathcal{E}) \simeq (\mathbb{P}^n, T_{\mathbb{P}^n})$.
\end{theorem}

It is worth noting that in 2006 C. Araujo developed a different approach to Theorem \ref{T:AW01} using the variety of minimal rational tangents \cite{MR2232023}. In 1973, S. Kobayashi and T. Ochiai proved the following theorem characterizing both projective spaces and quadric hypersurfaces: 

\begin{theorem}\label{T:KO}\cite{MR0316745}
Let $X$ be an $n$-dimensional compact complex manifold with ample line bundle $\mathcal{L}$. If $c_1(X) \geq (n+1)c_1(\mathcal{L})$ then $X \simeq \mathbb{P}^n$. If $c_1(X) = nc_1(\mathcal{L})$ then $X \simeq Q_n$, where $Q_n \subseteq \mathbb{P}^{n+1}$ is a hyperquadric.
\end{theorem}

Most recently, the following conjecture of A. Beauville \cite{MR1738060} was verified by Araujo, Druel, and Kov\'acs:

\begin{theorem}\label{T:ADK}\cite{MR2439607}
Let $X$ be a smooth complex projective variety of dimension $n$, and let $\mathcal{L}$ be an ample line bundle on $X$. If $H^0(X, \wedge^p T_X \otimes \mathcal{L}^{-p}) \neq 0$ for some positive integer $p$, then either $(X, \mathcal{L}) \simeq (\mathbb{P}^n, \mathcal{O}_{\mathbb{P}^n}(1))$ or $p = n$ and $(X, \mathcal{L}) \simeq (Q_p, \mathcal{O}_{Q_p}(1))$,  where $Q_p$ denotes a smooth quadric hypersurface in $\mathbb{P}^{p+1}$.
\end{theorem}

Theorems \ref{T:Mori}-\ref{T:ADK} are comparable in their direction but incongruous in the sense that no one of them implies all the others. Conjecture \ref{Q:main} is appealing since it simultaneously implies all of them: Mori's theorem is covered by the case $p = 1$, $\mathcal{E} = T_X$, Wahl's theorem by $p = 1$, $r = 1$, and the result of Andreatta-Wi\'sniewski by taking $p = 1$. The main theorem of \cite{MR2439607} is covered by setting $\mathcal{E} = \mathcal{L}^{\oplus r}$ where $r = p$, and \cite{MR0316745} by setting $\mathcal{E} = \mathcal{L}^{\oplus n}$ and $\mathcal{E} = \mathcal{L}^{\oplus n-1} \oplus \mathcal{L}^{\otimes 2}$.  

\begin{remark}
Notice that \ref{Q:main} also generalizes \ref{T:ADK} to the case where $\wedge^p T_X$ contains a product of $p$ distinct ample line bundles: $(\mathcal{L}_1 \otimes \mathcal{L}_2 \otimes \dots \otimes \mathcal{L}_p) \subseteq \wedge^p T_X$.
\end{remark}

It is easy to check that Conjecture \ref{Q:main} holds in some simple cases, for example, when the dimension of $X$ is small: If $\dim X = 1$, the only choice for the integer $p$ is $p = 1$. In this case, Conjecture \ref{Q:main} follows from Theorem \ref{T:Wahl} (and also Theorem \ref{T:AW01}.) When $\dim X = 2$, Conjecture \ref{Q:main} follows easily from the following theorem:

\begin{theorem}\label{T:dim2Helper}
Let $X$ be a smooth complex projective variety of dimension $2$, and assume that $-K_X = A + F$ where $F$ is an effective divisor and $A$ is an ample divisor such that $A \cdot C \geq 2$ for every smooth rational curve $C \subseteq X$, $C \simeq \mathbb{P}^1$. Then either $X \simeq \mathbb{P}^2$ or $X \simeq \mathbb{P}^1 \times \mathbb{P}^1$.
\end{theorem}

\begin{proof}
First notice that $X$ has negative Kodaira dimension since $-K_X \cdot C > 0$ for every general curve $C \subseteq X$. Let $X \longrightarrow X_{\text{min}}$ be a minimal model obtained by blowing down sufficiently many $(-1)$-curves. Since $\kappa(X) < 0$, $X_{\text{min}}$ is isomorphic to either $\mathbb{P}^2$ or a ruled surface over a curve $B$. Before addressing each case, I prove the following claim that will be used in the rest of the proof:

\begin{claim}\label{C:inF}
Let $X$, $F$, and $A$ be as in the statement of Theorem \ref{T:dim2Helper} above. If $C \subseteq X$ is a curve such that $C \simeq \mathbb{P}^1$ and $C^2 < 0$, then $F \cdot C < 0$ and hence $C \subseteq F$.
\end{claim}

\begin{proof}
The following computation implies the claim:
\[
F \cdot C = (-K_X - A) \cdot C = (-K_X \cdot C) - (A \cdot C) \leq (2 + C^2) - 2 < 0
\]
Here the first inequality follows from adjunction and the initial assumption on the ample divisor $A$.
\end{proof}

Continuing with the proof of Theorem \ref{T:dim2Helper}, assume that $X \not\simeq \mathbb{P}^2$. It follows that $X$ admits a morphism to a ruled surface $Y \longrightarrow B$: If $X_{\text{min}} \simeq \mathbb{P}^2$ then $Y$ is the blow-up of $\mathbb{P}^2$ at a single point. Otherwise $Y \simeq X_{\text{min}}$. The ruling $Y \longrightarrow B$ induces a morphism $\pi: X \longrightarrow B$. I will show by contradiction that the fibers of $\pi$ are irreducible, hence $X$ itself is ruled: Suppose that $G$ is a reducible fiber of $\pi$. Then $G$ may be written as a sum $G = \Sigma G_i$ where $G_i \simeq \mathbb{P}^1$ and $G_i^2 < 0$. By \ref{C:inF}, each $G_i$ (and hence $G$) is contained in the effective divisor $F$. Also, as $G$ is a fiber, $G \cdot G_i = 0$. It follows from \ref{C:inF} that $(F - G) \cdot G_i < 0$ for each $G_i$, therefore $G$ must be contained in $F - G$, i.e., $F$ contains $2G$. Repeating this computation, one may show that $nG \subseteq F$ for any positive integer $n$, but this is a contradiction since $F$ is a fixed effective divisor. Therefore the fibers of $\pi$ are irreducible as claimed, and $\pi: X \longrightarrow B$ itself must be a ruling of $X$.

Using the notation of \cite[V.2.8]{MR0463157}, there exists a distinguished locally free sheaf $\mathcal{E}^\prime$ of rank 2 and degree $-e$ such that $X \simeq \mathbb{P}(\mathcal{E}^{\prime})$. Furthermore, in this case there is a section $\sigma: B \longrightarrow X$ with image $C_0$ such that $\mathcal{L}(C_0) \simeq \mathcal{O}_{\mathbb{P}(\mathcal{E}^\prime)}(1)$. Continuing with the notation of \cite[V.2]{MR0463157}, let $\mathfrak{f}$ be a fiber of $\pi$. In particular, recall that $C_0 \cdot \mathfrak{f} = 1$ and $\mathfrak{f}^2 = 0$. By the assumption on $A$ and the fact that $\mathfrak{f}$ is nef, one has: $-K_X \cdot \mathfrak{f} = A \cdot \mathfrak{f} + F \cdot \mathfrak{f} \geq 2$ . On the other hand, by \cite[V.2.11]{MR0463157}, $-K_X \cdot \mathfrak{f} = 2$. Therefore $A \cdot \mathfrak{f} = 2$ and $F \cdot \mathfrak{f} = 0$, and the latter inequality implies that $F = m\mathfrak{f}$ is nef. It follows that $-K_X$ is ample, (it is the sum of an ample and a nef divisor), and therefore $X$ is a Del Pezzo surface. This means that $X$ is both ruled and rational, hence it is a Hirzebruch surface, i.e., $\mathcal{E}^\prime$ is decomposible. By \cite[2.12]{MR0463157}, it follows that $e \geq 0$. On the other hand, since $C_0 \nsubseteq F$, \ref{C:inF} implies that $C_0^2 \geq 0$. But $e = -C_0^2$ by \cite[V.2.9]{MR0463157}, therefore $e = C_0^2 = 0$. The only Hirzebruch surface with $e = 0$ is $\mathbb{P}^1 \times \mathbb{P}^1$, and this completes the proof of Theorem \ref{T:dim2Helper}.
\end{proof}

\begin{corollary}\label{C:dim2Case}
Conjecture \ref{Q:main} holds when $\dim X = 2$.
\end{corollary}

\begin{proof}If $\dim X = 2$, there are two choices for the integer $p$. If $p = 1$, Conjecture \ref{Q:main} follows from Theorem \ref{T:Wahl}, so we may assume that $p = 2$. Over a field of characteristic zero, the wedge product of an ample vector bundle is again ample \cite[5.3]{MR0193092}, so the condition $\wedge^2 \mathcal{E} \subseteq \wedge^2 T_X$ implies that $\omega_X^{-1}$ contains an ample line bundle. In particular, one may write $-K_X = A + F$ where $A = \mbox{c}_1(\wedge^2 \mathcal{E})$ is the correspondng ample divisor and $F$ is an effective divisor. Notice that $A \cdot C \geq 2$ for every smooth rational curve $C \subseteq X$, $C \simeq \mathbb{P}^1$: Since $\mathcal{E}$ is ample, the degree of $\mathcal{E}|_C = A|_C$ is bounded below by the rank of $\mathcal{E}$. Now Theorem 1.7 shows that Conjecture \ref{Q:main} holds when $\dim X = 2$.
\end{proof} 

% REMARK 1.10 (Conjecture holds from previous work when p = n.)
%\begin{remark}The argument given above can be generalized to the case $p = n = \dim X$. The point is that when $p$ is ``maximal'' in this sense, Conjecture \ref{Q:main} is implied by the recent results \cite{MR1929792} and \cite{MR2112591}. The first paper settles a conjecture of S. Mori and F. Mukai that claims $\mathbb{P}^n$ is characterized by the property that $-K_X \cdot C \geq n+1$ for all proper rational curves $C \subseteq X$. The second paper gives a similar characterization of smooth quadric hypersurfaces where one assumes instead that the lower bound is $n$. 
%\end{remark}

In this paper I will show that Conjecture \ref{Q:main} holds for all varieties with Picard number 1. The paper is organized as follows: Section 2 is devoted to gathering necessary definitions and results about minimal covering families of rational curves. Section 3 will cover some auxillary results needed for the main proof. The proof of Theorem \ref{T:main} is covered in Section 4.  
\newline
\newline
\textit{Notation:} I will follow the notation of \cite{MR1440180} in the discussion of rational curves. By a vector bundle I mean a locally free sheaf; a line bundle is an invertible sheaf. I will denote by $\mathbb{P}(V)$ the natural projectivization of a vector space $V$. A point $x \in X$ is general if it is contained in a dense open subset of $U \subseteq X$ where $U$ is a fixed open subset determined by the context. Throughout the paper I will be working over the field of complex numbers.
\newline
\newline
\textit{Acknowledgments:} I am immensely grateful to my advisor, \textit{S\'andor Kov\'acs}, for his attention, guidance, and many insights. I would also like to thank \textit{Carolina Araujo} for very helpful discussions and suggestions that improved the content of this paper. 
\newline
\newline
\textit{Note:} Upon completion of this paper, I learned of a somewhat related result by Matthieu Paris \cite{paris-2010}.

\section{Rational Curves of Minimal Degree on Uniruled Varieties}\label{S:Background}

The proof of the main theorem relies on studying rational curves of minimal degree on $X$. Starting with \cite{MR554387}, many tools have been developed for analyzing families of rational curves on uniruled varieties; for the reader's convenience I summarize the most important developments here. 

Let $X$ be a smooth complex projective variety. If $X$ is uniruled, one can find an irreducible component $H \subset \mbox{RatCurves}^n(X)$ such that the natural map $\mbox{Univ}_H \longrightarrow X$ is dominant. Such a component is called a \textit{dominating family} of rational curves on $X$. The component $H$ is called \textit{unsplit} if it is proper, and is called \textit{minimal} if the subfamily of curves parameterized by $H$ passing through a general point $x \in X$ is proper. A uniruled variety always admits a minimal dominating family of curves \cite[IV.2.4]{MR1440180}.

If $C \subset X$ is a rational curve on $X$ and $f: \mathbb{P}^1 \longrightarrow C \subseteq X$ is its normalization, the corresponding point in $\mbox{RatCurves}^n(X)$ is denoted by $[f]$. If $H$ is a minimal dominating family, then the splitting type of $f^\ast T_X$ for any general $[f] \in H$ is: 
\[
f^\ast T_X \simeq \splitTX
\]
where $d := \mbox{deg}(f^\ast T_X) - 2 \geq 0$ \cite[IV.2.9, IV.2.10]{MR1440180}. The ``positive part'' of $f^\ast T_X$ is the subbundle defined by:
\[
(f^\ast T_X)^+ := \mbox{im}[H^0(\mathbb{P}^1, f^\ast T_X(-1))\otimes\mathcal{O}_{\mathbb{P}^1}(1) \rightarrow f^\ast T_X] \hookrightarrow f^\ast T_X.
\]

If $H$ is a fixed minimal dominating family of rational curves on $X$, one can define an equivalence relation on the points of $X$ via $H$: Two points $x_1, x_2 \in X$ are $H$\textit{-equivalent} if they can be connected by a chain of rational curves parameterized by $H$. By \cite[IV.4.16]{MR1440180}, there exists a proper surjective morphism $\pi^\circ: X^\circ \longrightarrow Y^\circ$ from a dense open subset $X^\circ \subseteq X$ onto a normal variety $Y^\circ$ whose fibers are $H$-equivalence classes. The morphism $\pi^\circ$ is often called the $H$\textit{-rationally connected quotient} of $X$. If $Y^\circ$ is a point, then $X$ is called $H$\textit{-rationally connected}. An important fact used later is that when the Picard number of $X$ is 1, the $H$-rationally connected quotient is trivial:

\begin{proposition}\label{P:trivMRC}
Let $X$ be a smooth complex projective variety, $H$ a minimal dominating family of rational curves on $X$, and $\pi^\circ: X^\circ \longrightarrow Y^\circ$ the corresponding $H$-rationally connected quotient. If $\rho (X) = 1$, then $Y^\circ$ is a point.
\end{proposition}

\begin{proof}
Suppose that $Y^\circ$ is positive dimensional. Let $D_{Y^\circ}$ be an ample effective divisor on $Y^\circ$, $D_{X^\circ}$ its pullback on $X^\circ$ and $D_X$ the closure of $D_{X^\circ}$ in $X$. Since $\rho(X) = 1$, every effective divisor is ample, and it follows that every rational curve parameterized by $H$ has positive intersection with $D_X$. Let $C$ be a rational curve parameterized by $H$ and contained in $X^\circ$. By definition, $\pi^\circ$ contracts $C$ and hence $D_{X^\circ} \cdot C = 0$, a contradiction. Therefore $Y^\circ$ must be a point. 
\end{proof}

\begin{remark} 
The converse of Proposition \ref{P:trivMRC} is also true by \cite[IV.3.13.3]{MR1440180} if one assumes additionally that $H$ is unsplit, but this will not be needed here. 
\end{remark}

\begin{remark}
The equivalence relation above can be extended to a collection of families of rational curves $H_1, H_2, \dots, H_k$: Two points $x_1, x_2 \in X$ are $(H_1, H_2, \dots, H_k)$\textit{-equivalent} if they can be connected by a chain of rational curves parameterized by $H_1, H_2, \dots, H_k$. This induces a morphism on a dense open subset of $X$ with $(H_1, H_2, \dots, H_k)$-rationally connected fibers, called the $(H_1, H_2, \dots, H_k)$-rationally connected quotient of $X$.
\end{remark}

It is worth noting that a minimal dominating family $H$ may not always restrict to a minimal dominating family on the fibers of the $H$-rationally connected quotient. To be precise, if $X_y$ is a fiber of an $H$-rationally connected quotient of $X$ and $\iota$ is the natural map
\begin{equation}\label{ratMap}
\iota: \mbox{RatCurves}^n(X_y) \hookrightarrow \mbox{RatCurves}^n(X) 
\end{equation}
it is not always the case that $\iota^{-1}(H) \subseteq \mbox{RatCurves}^n(X_y)$ is irreducible:

\begin{example}\label{E:CounterEx} Let $Y \subseteq \mathbb{P}^9$ be the open subset parameterizing smooth quadric surfaces in $\mathbb{P}^3$, $X$ the corresponding open subset of the universal hypersurface in $\mathbb{P}^3 \times \mathbb{P}^9$, $\pi_1: X \longrightarrow \mathbb{P}^3$ and $\pi_2: X \longrightarrow Y \subseteq \mathbb{P}^9$ the restrictions of the usual projection morphisms. Let $C$ be a rational curve on $X$ corresponding to a line on a smooth quadric in $\mathbb{P}^3$. (In other words, $C$ has the property of being contracted by $\pi_2$ and having image equal to a line under $\pi_1$.) Let $H \subseteq \mbox{RatCurves}^n(X)$ be the irreducible component containing the point parameterizing $C$. 

I claim that $H$ is in fact a dominating family on $X$: First notice that $H$ parameterizes all the rational curves in $X$ that correspond to a line on a smooth quadric in $\mathbb{P}^3$. Indeed, if $C^\prime$ is any other rational curve with these properties, there exists a smooth deformation of $C$ to $C^\prime$ in $X$: The images of $C$ and $C^\prime$ in $\mathbb{P}^3$ are lines, say $L$ and $L^\prime$, and in $\mathbb{P}^3$ there exists a smooth deformation of $L$ to $L^\prime$ by a family of lines $\{L_t\}$ parameterized by $\mathbb{P}^1$. One can extend this to a family of smooth quadrics $\{Q_t\}$ parameterized over the same base such that $L_t \subset Q_t$ for each $t \in \mathbb{P}^1$. (For example, let $Q$ be the image of $\mathbb{P}^1 \times \mathbb{P}^1$ under the Segre embedding, and let $L$ be a distinguished line on $Q$. There exists a one-parameter family of automorphisms $\{\alpha_t\}$ of $\mathbb{P}^3$ such that $\alpha_t(L) = L_t$ for each $t \in \mathbb{P}^1$, (just choose an  appropriate non-trivial morphism $\mathbb{P}^1 \longrightarrow \mbox{Aut}(\mathbb{P}^3)$), and now the family $\{Q_t := \alpha_t(Q) \phantom{i}|\phantom{i} t \in \mathbb{P}^1\}$ has the desired properties.) Since $X$ is covered by the rational curves corresponding to the lines on the smooth quadrics of $\mathbb{P}^3$, $H$ is a dominating family on $X$.

Next notice that the $H$-rationally connected quotient is just $\pi_2: X \longrightarrow Y$: On one hand, by construction, every rational curve parameterized by $H$ is contained in a fiber of $\pi_2$. On the other hand, the fibers of $\pi_2$ are just the smooth quadrics in $\mathbb{P}^3$ and each is rationally connected by the lines it contains.

Finally, observe that the restriction of $H$ to any fiber cannot be a minimal dominating family: There are two minimal dominating families on any $\mathbb{P}^1 \times \mathbb{P}^1$, (namely the two families of lines), and the restriction of $H$ to any fiber will contain both of them.
\end{example}

\begin{remark}\label{R:Pic1} The above example also shows that one cannot assume in general that the fibers of the $H$-rationally connected quotient have Picard number 1, even when $H$ is unsplit. A necessary condition on $H$ for the fibers to have Picard number 1 is given by \cite[2.3]{MR2439607}.
\end{remark}

Next, recall the definition of the \textit{variety of minimal rational tangents}: If $x \in X$ is a general point of $X$, let $H_x$ denote the normalization of the subscheme of $H$ parameterizing curves passing through $x \in X$. For general $x \in X$, $H_x$ is a smooth projective variety of dimension $d := \mbox{deg}(f^\ast T_X) - 2$ \cite[II.1.7, II.2.16]{MR1440180}. There exists a map $\tau_x: H_x \dashrightarrow \mathbb{P}(T_xX)$ called the \textit{tangent map} defined by sending a curve that is smooth at $x \in X$ to its corresponding tangent direction at $x$. The closure of the image of $\tau_x$ in $\mathbb{P}(T_xX)$ is called the \textit{variety of minimal rational tangents} at $x$ and is denoted $\mathcal{C}_x \subseteq \mathbb{P}(T_xX)$. The tangent map is actually the normalization morphism of $\mathcal{C}_x$, a fact proved by S. Kebekus \cite{MR1874114} and J. Hwang  and N. Mok \cite{MR2128297}:

\begin{theorem}\label{T:CxNormal} \hspace*{1em}

\begin{enumerate-p}
\item \cite{MR1874114} The tangent map $\tau_x: H_x \dashrightarrow \mathcal{C}_x$ is a finite morphism.

\item \cite{MR2128297} The tangent map $\tau_x: H_x \dashrightarrow \mathcal{C}_x$ is birational, hence it is the normalization.
\end{enumerate-p}
\end{theorem}

The variety $\mathcal{C}_x$ has a natural embedding into $\mathbb{P}(T_xX)$, and this embedding yields important geometric information about $X$. For example, Araujo shows that when $\mathcal{C}_x$ is a linear subspace of $\mathbb{P}(T_xX)$, the $H$-rationally connected quotient of $X$ is a projective space bundle:

\begin{theorem}\label{T:AraMain}
\cite[1.1]{MR2232023} Assume that $\mathcal{C}_x$ is a $d$-dimensional linear subspace of $\mathbb{P}(T_xX)$ for a general point $x \in X$. Then there is a dense open subset $X^\circ$ of $X$ and $\mathbb{P}^{d+1}$-bundle $\varphi^\circ: X^\circ \longrightarrow T^\circ$ such that any curve from $H$ meeting $X^\circ$ is a line on a fiber of $\varphi^\circ$.
\end{theorem}

Lastly, note that the tangent space of $\mathcal{C}_x$ at a point $\tau_x([f])$ is related to the splitting type of $f^\ast T_X$ in an important way. In particular, the tangent space of $\mathcal{C}_x$ at the point $\tau_x([f])$ is cut out by the positive directions of $f^\ast T_X$ at $x \in X$:

\begin{lemma}\label{L:Hwa}
\cite[2.3]{MR1919462} \textit{Let $[f] \in H$ be a general member, and let $T_xX^+_f \subseteq T_xX$ be the $(d + 1)$-dimensional subspace corresponding to the positive factors of the splitting $f^\ast T_X \simeq$ \splitTX . Then $\mathbb{P}(T_xX^+_f)$ is the projectivized tangent space of $\mathcal{C}_x$ at the point $\tau_x([f])$}. 
\end{lemma}

\section{Preliminary Results}\label{S:Prelims}

Before proving the main theorem, I prove a few auxillary results. In particular, I will show that with the assumptions made in the statement of Theorem \ref{T:main},  $X$ admits a nice cover of rational curves, and one can determine the splitting type of the ample vector bundle $\mathcal{E}$ when restricted to these rational curves. 

\begin{lemma}\label{L:Xuniruled}
Let $X$ be a smooth complex projective variety, $\mathcal{E}$ an ample vector bundle of rank $r$ on $X$, and assume that $\wedge^p \mathcal{E} \subseteq \wedge^p T_X$ for some positive integer $p \leq r$. Then $X$ is uniruled.
\end{lemma}

\begin{proof}
Uniruledness of $X$ follows almost immediately from a theorem of Miyaoka, that says that if $\Omega_X$ is not generically semipositive, then $X$ is uniruled \cite[8.6]{MR927960}. Since generic semipositivity of $\Omega_X$ implies generic semipositivity of $\wedge^p \Omega_X$, it is enough to check that $\wedge^p \Omega_X$ is not generically semipositive: Let $C$ be a general complete intersection curve on $X$. Then $(\wedge^p \mathcal{E})|_{C}$ has positive degree since $\wedge^p \mathcal{E}$ is ample. The dual of the inclusion $(\wedge^p \mathcal{E})|_C \hookrightarrow (\wedge^p T_X)|_C$ yields the desired result. 
\end{proof}

Now let $H \subset \mbox{RatCurves}^n(X)$ be a minimal dominating family of rational curves on $X$ guaranteed by Lemma \ref{L:Xuniruled}. The next lemma determines the splitting type of $f^\ast \mathcal{E}$ for $[f] \in H$. 

\begin{lemma}\label{L:splitTX}
Let $X$ be a smooth complex projective variety, $\mathcal{E}$ an ample vector bundle of rank $r$ on $X$, and $p \leq r$ a positive integer such that $\wedge^p \mathcal{E} \subseteq \wedge^p T_X$. Let $H$ be a minimal dominating family of rational curves on $X$. Then either $f^\ast \mathcal{E} \simeq \mathcal{O}_{\mathbb{P}^1}(2) \oplus \mathcal{O}_{\mathbb{P}^1}(1)^{\oplus r-1}$ for every $[f] \in H$, or $f^\ast \mathcal{E} \simeq \mathcal{O}_{\mathbb{P}^1}(1)^{\oplus r}$ for every $[f] \in H$.
\end{lemma}

\begin{proof}
First let $[f] \in H$ be a general member of $H$. Since $\mathcal{E}$ is ample and $[f]$ parameterizes a rational curve, $f^\ast\mathcal{E}$ splits as a direct sum of positive degree line bundles:
\[
f^\ast \mathcal{E} \simeq \displaystyle \bigoplus_{i = 1}^{r} \mathcal{O}_{\mathbb{P}^1}(\alpha_i),\hspace{5mm} \alpha_i \geq 1.
\]
It follows that $f^\ast(\wedge^p\mathcal{E})$ splits as a sum of line bundles of degree at least $p$: 
\[
f^\ast(\wedge^p\mathcal{E}) \simeq \displaystyle \bigoplus_{j = 1}^{\binom{r}{p}} \mathcal{O}_{\mathbb{P}^1}(\beta_j),\hspace{5mm} \beta_j =  \alpha_{j_1} + \alpha_{j_2} + \dots + \alpha_{j_p} \geq p.
\]
By assumption, 
\begin{multline*}
f^\ast(\wedge^p\mathcal{E}) \subseteq f^\ast(\wedge^p T_X) \simeq \wedge^p(\splitTX) \\ \simeq \mathcal{O}_{\mathbb{P}^1}(p+1)^{\oplus q_1} \oplus \mathcal{O}_{\mathbb{P}^1}(p)^{\oplus q_2}\oplus\dots
\end{multline*}
and the highest degree line bundle occuring on the right is $\mathcal{O}_{\mathbb{P}^1}(p+1)$. Therefore $p \leq \beta_j \leq p+1$ for each $1 \leq j \leq \binom{r}{p}$, but this leaves only two possibilities for $f^\ast\mathcal{E}$: Either $f^\ast \mathcal{E} \simeq \mathcal{O}_{\mathbb{P}^1}(2) \oplus \mathcal{O}_{\mathbb{P}^1}(1)^{\oplus r-1}$ or $f^\ast \mathcal{E} \simeq \mathcal{O}_{\mathbb{P}^1}(1)^{\oplus r}$. 

Lastly, observe that $\mathcal{E}$ must split the same way on every rational curve parameterized by $H$: Since $H$ is an irreducible component of $\mbox{RatCurves}^n(X)$, the intersection number of a fixed line bundle on $X$ and any curve $C$ parameterized by $H$ is independent of $C$. In particular, the degree of $\mbox{det}(\mathcal{E})$ remains constant on all the rational curves parameterized by $H$, and it follows that $\mbox{deg}(f^\ast \mathcal{E}) = r$ for every $[f] \in H$ or $\mbox{deg}(f^\ast \mathcal{E}) = r+1$ for every $[f] \in H$. That $\mathcal{E}$ splits in one of the above two ways on every (i.e., not just general) $[f] \in H$ is forced by the fact that $f^\ast\mathcal{E}$ is ample and its rank and degree differ by at most 1. 
\end{proof}

\begin{corollary}\label{C:unsplit}
Let $X$ and $\mathcal{E}$ be as above. Unless $r = 1$ and $f^\ast \mathcal{E} \simeq \mathcal{O}_{\mathbb{P}^1}(2)$, $X$ admits an unsplit minimal dominating family of rational curves.
\end{corollary}

\begin{proof}
Let $H$ be a minimal dominating family of rational curves on $X$, $[f] \in H$ a general member. By Lemma \ref{L:splitTX},
\[
r = \mbox{rank}(f^\ast\mathcal{E}) \leq \mbox{deg}(f^\ast\mathcal{E}) = r \mbox{ or } r+1
\]
When $r > 1$ or when $r = 1$ and $f^\ast \mathcal{E} \simeq \mathcal{O}_{\mathbb{P}^1}(1)$, the above inequality shows that it is impossible for the curve parameterized by $[f]$ to split as a sum of two or more rational curves $C_1, C_2, \dots, C_k$: On the one hand $r = \mbox{rank}(f^\ast \mathcal{E}) \leq \mbox{deg}(\mathcal{E}|_{C_i})$ for each $C_i$ by ampleness of $\mathcal{E}$. On the other hand, the sum of the degrees of the $\mathcal{E}|_{C_i}$ must equal $r \mbox{ or } r+1$. Therefore $H$ is unsplit.
\end{proof}

\section{Proof of Theorem \ref{T:main}}\label{S:meat}

Let $X$ be a smooth complex projective variety of dimension $n$, $\mathcal{E}$ an ample vector bundle of rank $r$ on $X$, and $p \leq r$ a positive integer such that $\wedge^p \mathcal{E} \subseteq \wedge^p T_X$. By Theorem \ref{T:Wahl}, one may assume that $r > 1$. Let $H$ be an unsplit minimal dominating family of rational curves on $X$ guaranteed by Corollary \ref{C:unsplit}. Lemma \ref{L:splitTX} shows that there are two possible ways for the vector bundle $\mathcal{E}$ to split on the curves parameterized by $H$; I address each case separately:
\newline
\newline
CASE I: First assume that $f^\ast \mathcal{E} \simeq \mathcal{O}_{\mathbb{P}^1}(1)^{\oplus r}$ for every $[f] \in H$. The following result of Andreatta-Wi\'sniewski deals with this situation:

\begin{theorem}\label{T:AW1.2} \cite[1.2]{MR1859022} Let $X$ be a smooth complex projective variety such that $\rho(X) = 1$, $\mathcal{E}$ a vector bundle of rank $r$ on $X$, and $H$ an unsplit minimal dominating family of rational curves on $X$. If there exists an integer $a$ such that $f^\ast \mathcal{E} \simeq \mathcal{O}_{\mathbb{P}^1}(a)^{\oplus r}$ for every $[f] \in H$, then there is a uniquely defined line bundle $\mathcal{L}$ on $X$ such that $deg(f^\ast \mathcal{L}) = a$ and $\mathcal{E} \simeq \mathcal{L}^{\oplus r}$.
\end{theorem}

\begin{remark}I was unable to follow all of the argument made in \cite[1.2]{MR1859022}, therefore an alternative proof is provided below. The method of lifting rational curves to $\mathbb{P}(\mathcal{E})$ remains the same as the proof given in \cite{MR1859022}; modifications were made to reflect the fact that a general fiber of a rationally connected quotient may not have Picard number 1. (See \ref{E:CounterEx}-\ref{R:Pic1} for more.) In fact, Theorem \ref{T:sum} is a generalization of the original statement. Since then, M. Andreatta has explained to me a nice fix for the apparent gap in the original proof of \cite[1.2]{MR1859022}.
\end{remark}

\begin{theorem}\label{T:AWfix}Let $X$ be a smooth complex projective variety of dimension $n$, $\mathcal{E}$ a vector bundle of rank $r$ on $X$, and $H_1, H_2, \dots, H_k$ a collection of families of rational curves on $X$ such that $X$ is $(H_1, H_2, \dots, H_k)$-rationally connected. If there exists an integer $a \in \mathbb{Z}$ such that $f^\ast \mathcal{E} \simeq \mathcal{O}_{\mathbb{P}^1}(a)^{\oplus r}$ for every $[f] \in H_1, H_2, \dots, H_k$, then there exists a finite surjective morphism $q: Y \longrightarrow X$ from a variety $Y$ such that:
\begin{enumerate-p}
\item There is a collection of families $V_1, V_2, \dots, V_l \subseteq \RatCurves^n(Y)$ and a proper surjective morphism $q_\ast: \bigcup_{i = 1}^l V_i \longrightarrow \bigcup_{j = 1}^k H_k$ where $q_\ast([\widehat{f}]) = [f]$ is given by $q \circ \widehat{f} = f$. The variety $Y$ is $(V_1, V_2, \dots, V_l)$-rationally connected. \label{condition1}

\item There is a (uniquely defined) line bundle $\mathcal{L}$ on $Y$ such that $\degree(f^\ast \mathcal{L}) = a$ and $q^\ast\mathcal{E} \simeq \mathcal{L}^{\oplus r}$. \label{condition2}
\end{enumerate-p}
\end{theorem}  

\begin{proof}
The argument applies induction with respect to $r$. Let $p: \mathbb{P}(\mathcal{E}) \longrightarrow X$ be the projectivization of $\mathcal{E}$ with relative tautological bundle $\mathcal{O}_{\mathbb{P}(\mathcal{E})}(1)$. For any $[f] \in H_1, H_2, \dots, H_k$ and $y \in p^{-1}(f(0))$ there is a unique lift $\widehat{f}: \mathbb{P}^1 \longrightarrow \mathbb{P}(\mathcal{E})$ with the property that $p \circ \widehat{f} = f$ and $\mbox{deg}(\widehat{f}^\ast\mathcal{O}_{\mathbb{P}(\mathcal{E})}(1)) = a$, $\widehat{f}(0) = y$: Since $\mathbb{P}(f^\ast\mathcal{E}) = \mathbb{P}^1 \times \mathbb{P}^{r-1}$, the morphism $\widehat{f}$ is obtained by composing $\mathbb{P}(f^\ast\mathcal{E}) \longrightarrow \mathbb{P}(\mathcal{E})$ with the morphism $\mathbb{P}^1 \longrightarrow \mathbb{P}^1 \times \{y\} \subset \mathbb{P}^1 \times \mathbb{P}^{r-1}$. Thus, for a generic $f$, $\widehat{f}^\ast T_{\mathbb{P}(\mathcal{E})} = f^\ast T_X  \oplus \mathcal{O}^{\oplus (r-1)}$. 

For each $1 \leq i \leq k$, one may choose an irreducible component $\widehat{H_i} \subset \mbox{RatCurves}^n(\mathbb{P}(\mathcal{E}))$ parameterizing these lifts such that $\widehat{H_i}$ dominates $H_i$. In fact, there exists a natural morphism $p_\ast: \mbox{RatCurves}^n(\mathbb{P}(\mathcal{E})) \longrightarrow \mbox{RatCurves}^n(X)$ defined by $p_\ast(\widehat{f}) = p \circ \widehat{f}$. 

\setcounter{claim}{\value{equation}}
\begin{claim}
For each $1 \leq i \leq k$, the morphism $p_*: \widehat{H_i} \longrightarrow H_i$ is proper and thus surjective. 
\end{claim}
\setcounter{equation}{\value{claim}}

\begin{proof} The proof uses the valuative criterion of properness \cite[II.4.7]{MR0463157}. Let $B$ be the spectrum of a discrete valuation ring (or a germ of a smooth curve in the analytic context) with a closed point $\delta$ and a general point $B^0$. Then for any family of morphisms $F_B : B \times \mathbb{P}^1 \longrightarrow X$ coming from $B \longrightarrow H_i$ one has $\mathbb{P}(F^\ast_{B}\mathcal{E}) = B \times \mathbb{P}^1 \times \mathbb{P}^{r-1}$. Now take $\widehat{F}_{B^{0}}: B^{0} \times \mathbb{P}^1 \longrightarrow \mathbb{P}(\mathcal{E})$, coming from a lift $B^0 \longrightarrow \widehat{H_i}$ of $B \longrightarrow H_i$. By construction $\widehat{F}_{B^{0}}$ is the composition of $\mathbb{P}(F^\ast_{B}\mathcal{E}) \longrightarrow \mathbb{P}(\mathcal{E})$ with the product $id \times \psi_{0}: B^0 \times \mathbb{P}^1 \longrightarrow (B^0 \times \mathbb{P}^1)\times \mathbb{P}^{r-1}$, for some constant morphism $\psi_{0}: B^{0} \longrightarrow \mathbb{P}^{r-1}$. The morphism $\psi_{0}$ extends trivially to $\psi: B \longrightarrow \mathbb{P}^{r-1}$, thus $\widehat{F}_{B^0}$ extends to $\widehat{F}_B$ which is the composition of $\mathbb{P}(F^\ast_{B}\mathcal{E}) \longrightarrow \mathbb{P}(\mathcal{E})$ with the product $id \times \psi$, hence $p_\ast$ is proper.
\end{proof}

Continuing the proof of Theorem \ref{T:AWfix}, consider the $(\widehat{H}_1, \widehat{H}_2, \dots, \widehat{H}_k)$-rationally connected quotient of $\mathbb{P}(\mathcal{E})$, and let $Y \subset \mathbb{P}(\mathcal{E})$ be a general fiber. Notice that $\widehat{H}_1, \widehat{H}_2, \dots, \widehat{H}_k$ restricts to a collection of families $\widehat{H}_{Y_1}, \widehat{H}_{Y_2}, \dots, \widehat{H}_{Y_m} \subseteq \text{RatCurves}^n(Y)$, and $Y$ is $(\widehat{H}_{Y_1}, \widehat{H}_{Y_2}, \dots, \widehat{H}_{Y_m})$-rationally connected by construction. Also note that $Y$ is projective and smooth.

Since $X$ is $(H_1, H_2, \dots, H_k)$-rationally connected and $p_\ast: \widehat{H}_i \longrightarrow H_i$ is surjective for each $1 \leq i \leq k$, the restriction map $p_Y: Y \longrightarrow X$ is surjective.

\setcounter{claim}{\value{equation}}
\begin{claim}
The morphism $p_Y$ has no positive dimensional fiber, hence it is a finite morphism.
\end{claim}
\setcounter{equation}{\value{claim}}

\begin{proof}
By \cite[II.4.4]{MR1440180}, the morphism $p: \mathbb{P}(\mathcal{E}) \longrightarrow X$ induces a surjective map
\begin{equation}
 A_1(\mathbb{P}(\mathcal{E}))_{\mathbb{Q}} \stackrel{p_\ast}{\longrightarrow} A_1(X)_{\mathbb{Q}} \longrightarrow 0. 
\end{equation}
Let $d$ be the dimension of  $A_1(X)_{\mathbb{Q}}$. Then the dimension of $A_1(\mathbb{P}(\mathcal{E}))_{\mathbb{Q}}$ is $d+1$ \cite[II.4.5]{MR1440180}, \cite[Ex. II.7.9]{MR0463157}, and the kernel of $p_\ast$ is the one dimensional space of 1-cycles in $A_1(\mathbb{P}(\mathcal{E}))_\mathbb{Q}$ that are contained in the fibers of $p$. Since the fibers of $p$ are projective spaces, these 1-cycles are each rationally equivalent to a line in a fiber of $p$. Therefore, since $X$ is rationally connected, they must be rationally equivalent in $A_1(\mathbb{P}(\mathcal{E}))_\mathbb{Q}$. If by contradiction there exists a proper curve $C \subset Y$ contracted by $p_Y$, one may take $C$ as a generator for the kernel of $p_\ast$.

Now by \cite[IV.3.13.3]{MR1440180}, $A_1(X)_\mathbb{Q}$ is generated by the classes of curves parameterized by $H_1, H_2, \dots, H_k$, and $A_1(Y)_\mathbb{Q}$ is generated by the classes of curves parameterized by $\widehat{H}_{Y_1}, \widehat{H}_{Y_2}, \dots, \widehat{H}_{Y_m}$. Therefore one may choose lifts of $d$ curves from $H_1, H_2, \dots, H_k$, say $\widehat{C}_1, \widehat{C}_2, \dots, \widehat{C}_d$, such that $\widehat{C}_i \subset Y$ for $1 \leq i \leq d$, and $A_1(\mathbb{P}(\mathcal{E}))_\mathbb{Q}$ is generated by $\widehat{C}_1, \widehat{C}_2, \dots, \widehat{C}_d$ and $C$. But $C \subset Y$ by assumption, so $C$ is a $\mathbb{Q}$-linear combination of $\widehat{C}_1, \widehat{C}_2, \dots, \widehat{C}_d$. This implies that $A_1(\mathbb{P}(\mathcal{E}))_\mathbb{Q}$ can be generated by $d$ elements, a contradiction. Therefore $p_Y$ does not contract any proper curve in $Y$, hence it is a finite morphism as desired.
\end{proof}

Now consider the pullback $\widetilde{p}: \mathbb{P}(p^\ast_Y\mathcal{E}) \longrightarrow Y$ with the induced morphism $\widetilde{p}_Y: \mathbb{P}(p^\ast_Y \mathcal{E}) \longrightarrow \mathbb{P}(\mathcal{E})$ such that $p \circ \widetilde{p}_Y = p_Y \circ \widetilde{p}$. By the universal property of the fiber product the projective bundle $\widetilde{p}$ admits a section $s : Y \longrightarrow \mathbb{P}(p^\ast_Y\mathcal{E})$ such that $\widetilde{p}_Y \circ s$ is the embedding of $Y$ into $\mathbb{P}(\mathcal{E})$. This induces a sequence of bundles over $Y$:
\begin{equation}\label{E:star}
0 \longrightarrow \mathcal{E}^\prime \longrightarrow p^\ast_Y\mathcal{E} \longrightarrow \mathcal{O}_{\mathbb{P}(\mathcal{E})}(1)|_Y \longrightarrow 0 
\end{equation}
where $\mathcal{E}^\prime$ is a bundle of rank $r-1$ on $Y$. In order to apply the inductive hypothesis to $\mathcal{E}^\prime$, it suffices to show that $\mathcal{E}^\prime$ splits in the desired way on the curves parameterized by $\widehat{H}_{Y_1}, \widehat{H}_{Y_2}, \dots, \widehat{H}_{Y_m}$: First notice that $\mbox{deg}(f^\ast\mathcal{O}_{\mathbb{P}(\mathcal{E})}(1)|_{Y}) =a $ for any curve $[f] \in \widehat{H}_{Y_1}, \widehat{H}_{Y_2}, \dots, \widehat{H}_{Y_m}$. (This follows from the fact that $\mbox{deg}(\widehat{f}^\ast \mathcal{O}_{\mathbb{P}(\mathcal{E})}(1)) = a$ for every $[\widehat{f}] \in \widehat{H}_1, \widehat{H}_2, \dots, \widehat{H}_k$ as stated at the beginning of the proof.) Therefore, by restricting \eqref{E:star} to any $[\widehat{f}] \in \widehat{H}_{Y_1}, \widehat{H}_{Y_2}, \dots, \widehat{H}_{Y_m}$, one has:
\[
0 \longrightarrow \widehat{f}^\ast\mathcal{E}^\prime \longrightarrow \mathcal{O}_{\mathbb{P}^1}(a)^{\oplus r} \longrightarrow \mathcal{O}_{\mathbb{P}^1}(a) \longrightarrow 0
\]
Twisting this sequence by $\mathcal{O}_{\mathbb{P}^1}(-a-1)$ yields:
\[
0 \longrightarrow \widehat{f}^\ast\mathcal{E}^\prime(-a-1) \longrightarrow \mathcal{O}_{\mathbb{P}^1}(-1)^{\oplus r} \longrightarrow \mathcal{O}_{\mathbb{P}^1}(-1) \longrightarrow 0
\]
Now, one may write $\widehat{f}^\ast\mathcal{E}^\prime(-a-1) \simeq \oplus_{i=1}^{r-1}\mathcal{O}_{\mathbb{P}^1}(\beta_i)$ where $\Sigma_{i=1}^{r-1}\beta_i = - (r - 1)$. The inclusion $\widehat{f}^\ast\mathcal{E}^\prime(-a-1) \hookrightarrow \mathcal{O}_{\mathbb{P}^1}(-1)^{\oplus r}$ implies  that $\widehat{f}^\ast\mathcal{E}^\prime(-a-1)$ has no global sections, hence $\beta_i < 0$ for $1 \leq i \leq r-1$. But since $\Sigma_{i=1}^{r-1}\beta_i = - (r - 1)$, $\beta_i = -1$ for all $1 \leq i \leq r-1$. It follows that $\widehat{f}^\ast\mathcal{E}^\prime \simeq \mathcal{O}_{\mathbb{P}^1}(a)^{\oplus r-1}$ for every $[\widehat{f}] \in \widehat{H}_{Y_1}, \widehat{H}_{Y_2}, \dots, \widehat{H}_{Y_m}$. Now let $q^\prime: Y^\prime \longrightarrow Y$ be the finite surjective morphism obtained from induction, and $V_1, V_2, \dots, V_l$ the corresponding collection of families of rational curves in $\text{RatCurves}^n(Y^\prime)$ satisfiying the conditions in \ref{condition1}. Pulling back the exact sequence \eqref{E:star} to $Y^\prime$ one obtains:
\begin{equation}\label{E:starstar}
0 \longrightarrow \mathcal{L}^{\oplus r-1} \longrightarrow (q^{\prime} \circ p_Y)^\ast \mathcal{E} \longrightarrow \mathcal{L}^\prime \longrightarrow 0 
\end{equation}
where $\mathcal{L}$ is the uniquely defined line bundle coming from induction, and $\mathcal{L}^\prime = q^{\prime\ast} \mathcal{O}_{\mathbb{P}(\mathcal{E})}(1)|_Y$ for simplicity. (Note that \eqref{E:starstar} is exact since the sheaves in \eqref{E:star} are each locally free.) I claim that $\mathcal{L} \simeq \mathcal{L}^\prime$ as line bundles on $Y^\prime$: First notice that $\mathcal{L}$ and $\mathcal{L}^\prime$ agree on all of the rational curves parameterized by $V_1, V_2, \dots, V_l$. Indeed, for any $[f^\prime] \in V_1, V_2, \dots, V_l$:
\[
f^{\prime\ast}\mathcal{L}^{\oplus r-1} = (f^\prime \circ q^{\prime})^\ast\mathcal{E} = \widehat{f}^\ast\mathcal{E} \simeq \mathcal{O}_{\mathbb{P}^1}(a)^{\oplus r-1}
\]
\hspace*{15em}and
\[
f^{\prime\ast}\mathcal{L}^\prime = (f^\prime \circ q^{\prime})^\ast\mathcal{O}_{\mathbb{P}^1}(1)|_Y = \widehat{f}^\ast \mathcal{O}_{\mathbb{P}^1}(1)|_Y \simeq \mathcal{O}_{\mathbb{P}^1}(a)
\]
 where $[\widehat{f}] \in \widehat{H}_{Y1}, \widehat{H}_{Y2}, \dots, \widehat{H}_{Ym}$ is the image of $[f^\prime]$ under the map $q_\ast^\prime$ given in \ref{condition1}. Now, $N_1(Y^\prime)$ is generated by the classes of curves coming from $V_1, V_2, \dots, V_l$ \cite[IV.3.13.3]{MR1440180} and there exists a nondegenerate bilinear pairing
\[
N_1(Y^\prime) \times N^1(Y^\prime) \longrightarrow \mathbb{Q}
\]
given by the intersection number of curves and divisors. Since the pairing is nondegenerate, it follows that $\mathcal{L}^{-1}\otimes\mathcal{L}^\prime$ is numerically equivalent to $\mathcal{O}_{Y^\prime}$,  and therefore $\mathcal{L}^{-1}\otimes\mathcal{L}^\prime$ is torsion \cite[1.1.20]{MR2095471}. Let $\mbox{Spec}(\mathcal{A}) \longrightarrow Y^\prime$ be the unramified cyclic cover of $Y^\prime$ induced by the $\mathcal{O}_{Y^\prime}$-algebra $\mathcal{A} = \mathcal{O}_{Y^\prime} \oplus (\mathcal{L}^{-1}\otimes\mathcal{L}^\prime) \oplus (\mathcal{L}^{-1}\otimes\mathcal{L}^\prime)^{\otimes 2} \oplus \dots \oplus (\mathcal{L}^{-1}\otimes\mathcal{L}^\prime)^{\otimes m-1}$, where $m$ is the smallest positive integer such that $(\mathcal{L}^{-1}\otimes\mathcal{L}^\prime)^{\otimes m} = \mathcal{O}_{Y^\prime}$. By the inductive assumption, $Y^\prime$ is rationally connected, hence simply connected, therefore $\mbox{Spec}(\mathcal{A}) \longrightarrow Y^\prime$ must be trivial. Therefore $m = 1$ and $(\mathcal{L}^{-1}\otimes\mathcal{L}^\prime) \simeq \mathcal{O}_{Y^\prime}$ as desired.

Lastly, since $Y^\prime$ is rationally connected, \cite[IV.3.8]{MR1440180} implies that $0 = H^0(Y^\prime, \Omega_{Y^\prime}^1) \simeq H^1(Y^\prime, \mathcal{O}_{Y^\prime})$, and therefore the sequence \ref{E:starstar} splits. In other words, $(q^\prime \circ p_Y)^\ast\mathcal{E} \simeq \mathcal{L}^{\oplus r}$ on $Y^\prime$, and this completes the proof of Theorem \ref{T:AWfix}.
\end{proof}

\begin{lemma}\label{L:AW1.2.2} \cite[1.2.2]{MR1859022} Let $X$ be a smooth complex projective Fano variety with $p: \mathbb{P}(\mathcal{E}) \longrightarrow X$ a projectivization of a rank $r$ bundle. Suppose that $\Psi: Y \longrightarrow X$ is a finite morphism. If $\mathbb{P}(\Psi^\ast(\mathcal{E})) \simeq Y \times \mathbb{P}^{r-1}$ then $\mathbb{P}(\mathcal{E}) \simeq X \times \mathbb{P}^{r-1}.$ $\Box$
\end{lemma}

\begin{theorem}\label{T:sum} Let $X$ be a smooth complex projective Fano variety, $\mathcal{E}$ a vector bundle of rank $r$ on $X$, and $H_1, H_2, \dots, H_k \subseteq \text{RatCurves}^n(X)$ a collection of rational curves such that $X$ is $(H_1, H_2, \dots, H_k)$-rationally connected. If there exists an integer $a \in \mathbb{Z}$ such that $f^\ast \mathcal{E} \simeq \mathcal{O}_{\mathbb{P}^1}(a)^{\oplus r}$ for every $[f] \in H_1, H_2, \dots, H_k$, then there is a uniquely defined line bundle $\mathcal{L}$ on $X$ such that $deg(f^\ast \mathcal{L}) = a$ and $\mathcal{E} \simeq \mathcal{L}^{\oplus r}$.
\end{theorem}

\begin{proof}
This is immediate from Theorem \ref{T:AWfix} and Lemma \ref{L:AW1.2.2}.
\end{proof}

\begin{remark}When $X$ is both uniruled and $\rho(X) = 1$, $X$ must be Fano. Therefore Theorem \ref{T:AW1.2} follows from Theorem \ref{T:sum}.
\end{remark}

Continuing with the proof of Theorem \ref{T:main}, use Theorem \ref{T:sum} to define a new vector bundle $\mathcal{F} := \mathcal{L}^{\oplus p}$ on $X$, and note that in our case $\mathcal{L}$ is ample. Recall that $p \leq r$ by assumption, therefore $\mathcal{F} \subseteq \mathcal{E}$. It follows that $\mathcal{L}^{\otimes p} = \mbox{det}(\mathcal{F}) \subseteq \wedge^p \mathcal{E} \subseteq \wedge^p T_X$. By \cite[6.3]{MR2439607} (= Theorem \ref{T:ADK}), either $X \simeq \mathbb{P}^n$ or $X \simeq Q_p \subseteq \mathbb{P}^{p+1}$.
\newline
\newline
CASE II: Now assume that $\rho(X) = 1$ and consider the case that $f^\ast \mathcal{E} \simeq \mathcal{O}_{\mathbb{P}^1}(2) \oplus \mathcal{O}_{\mathbb{P}^1}(1)^{\oplus r-1}$ for every $[f] \in H$. I will show that there exists a vector bundle injection $f^\ast \mathcal{E} \hookrightarrow f^\ast T_X^+$, and use this fact to study the geometry of the variety of minimal rational tangents $\mathcal{C}_x$ at general points $x \in X$. 

\begin{lemma}\label{L:vbInj}
Let $X$ be a smooth complex projective variety of dimension $n$, and let $\mathcal{E}$ be an ample vector bundle of rank $r$ on $X$ such that $\wedge^p \mathcal{E} \subseteq \wedge^p T_X$ for some positive integer $p \leq r$. Let $H$ be a minimal dominating family of rational curves on $X$, and let $[f] \in H$ be a general member. If $f^\ast \mathcal{E} \simeq \mathcal{O}_{\mathbb{P}^1}(2) \oplus \mathcal{O}_{\mathbb{P}^1}(1)^{\oplus r-1}$, then there exists a vector bundle injection $f^\ast \mathcal{E} \hookrightarrow f^\ast T_X^+$.
\end{lemma}

\begin{proof}
Since $[f] \in H$ is a general member, the splitting type of $T_X$ on the curve parameterized by $[f]$ is $f^\ast T_X \simeq$ \splitTX \cite[IV.2.9, IV.2.10]{MR1440180}. When $f^\ast \mathcal{E} \simeq \mathcal{O}_{\mathbb{P}^1}(2) \oplus \mathcal{O}_{\mathbb{P}^1}(1)^{\oplus r-1}$, a simple counting argument shows that $r - 1 \leq d$: If $f^\ast\mathcal{E} \simeq \mathcal{O}_{\mathbb{P}^1}(2) \oplus \mathcal{O}_{\mathbb{P}^1}(1)^{\oplus r-1}$ then $f^\ast(\wedge^p\mathcal{E})$ splits as a sum of line bundles of which exactly $\binom{r-1}{p-1}$ have degree $p + 1$. A similar computation shows that the direct sum decomposition of $f^\ast(\wedge^p T_X)$ includes exactly $\binom{d}{p-1}$ line bundles of degree $p+1$. Since $p + 1$ is the largest degree of any line bundle occuring in the decomposition of $f^\ast(\wedge^pT_X)$ and since I assume $f^\ast(\wedge^p\mathcal{E}) \subseteq f^\ast(\wedge^p T_X)$, it follows that $r - 1 \leq d$. But then $f^\ast\mathcal{E} \hookrightarrow f^\ast T_X^{+} \simeq\mathcal{O}_{\mathbb{P}^1}(2) \oplus \mathcal{O}_{\mathbb{P}^1}(1)^{\oplus d}$ as desired.
\end{proof}

Now let $x \in X$ be a general point, $H_x$ the normalization of the subscheme of $H$ parameterizing curves passing through $x \in X$, and $\tau_x: H_x \longrightarrow \mathcal{C}_x \subseteq \mathbb{P}(T_xX)$ the tangent map defined in Section 2. Let $H_x^i$, $1 \leq i \leq k$, be the irreducible components of $H_x$, and define $\mathcal{C}_x^i := \mbox{im}(\tau_x(H_x^i))$. Fix an irreducible component $H_x^i$ and let $[f] \in H_x^i$ be a general member $f: \mathbb{P}^1 \longrightarrow X$ such that $f(o) = x$ for a point $o \in \mathbb{P}^1$. The fiber $(f^\ast T_X)_o$ of $f^\ast T_X$ over the point $o$ is naturally isomorphic to $T_xX$, and under this isomorphism the positive part $(f^\ast T_X)^+_o \subset (f^\ast T_X)_o$ cuts out a $(d+1)$-dimensional linear subspace $T_xX_f^+ \subseteq T_xX$. By Lemma \ref{L:vbInj}, $(f^\ast \mathcal{E})_o \hookrightarrow (f^\ast T_X)_o^+$, and this induces the inclusion $\mathcal{E}_x \subseteq T_xX^+_f$. By \cite[2.3]{MR1919462} (= Lemma \ref{L:Hwa}), it follows that $\mathbb{P}(\mathcal{E}_x) \subseteq \overline{T_{\tau_x([f])}\mathcal{C}_x^i} \subseteq \mathbb{P}(T_xX)$. Now the argument in \cite[4.1, 4.2, 4.3]{MR2232023} implies that $\mathcal{C}_x^i$ is a linear subspace of $\mathbb{P}(T_xX)$; I include an outline of the main steps here for the convenience of the reader: By Lemma \ref{L:Ara4.2} below, the inclusion $\mathbb{P}(\mathcal{E}_x) \subseteq \overline{T_{\tau_x([f])}\mathcal{C}_x^i}$ forces $\mathcal{C}_x^i$ to have the structure of a cone in $\mathbb{P}(T_xX)$ with $\mathbb{P}(\mathcal{E}_x)$ contained in its vertex. Now the result follows from Lemma \ref{L:Ara4.3} and the fact that $H_x$ is smooth \cite[II.1.7, II.2.16]{MR1440180} and $\tau_x: H_x \longrightarrow \mathcal{C}_x$ is the normalization morphism (\cite{MR1874114}, \cite{MR2128297} = Theorem \ref{T:CxNormal}.)

\begin{lemma}\label{L:Ara4.2}
\cite[4.2]{MR2232023} Let $Z$ be an irreducible closed subvariety of $\mathbb{P}^m$. Assume that there is a dense open subset $U$ of the smooth locus of $Z$ and a point $z_0 \in \mathbb{P}^m$ such that $z_0 \in \bigcap_{z \in U}T_zZ$. Then $Z$ is a cone in $\mathbb{P}^m$ and $z_0$ lies in the vertex of this cone.
\end{lemma}

\begin{lemma}\label{L:Ara4.3}
\cite[4.3]{MR2232023} If $Z$ is an irreducible cone in $\mathbb{P}^m$ and the normalization of $Z$ is smooth, then $Z$ is a linear subspace of $\mathbb{P}^m$.
\end{lemma}

From here one can conclude that the irreducible components of $\mathcal{C}_x$ are all linear subspaces of $\mathbb{P}(T_xX)$. The following proposition of J.M. Hwang shows that in this case $\mathcal{C}_x$ is actually irreducible, thus itself a linear subspace of $\mathbb{P}(T_xX)$:

\begin{proposition}\label{P:Hwa2.2}
\cite[2.2]{MR1919462} Let $X$ be a smooth complex projective variety, $H$ a minimal dominating family of rational curves on $X$, and $\mathcal{C}_x \in \mathbb{P}(T_xX)$ the corresponding variety of minimal rational tangents at $x \in X$. Assume that for a general $x \in X$, $\mathcal{C}_x$ is a union of linear subspaces of $\mathbb{P}(T_xX)$. Then the intersection of any two irreducible components of $\mathcal{C}_x$ is empty.
\end{proposition} 

Now since $H_x$ is the normalization of $\mathcal{C}_x$ and it is dimension $d := \mbox{deg}(f^\ast T_X) - 2$ for a general point $x \in X$, $\mathcal{C}_x$ is in fact a linear subspace of $\mathbb{P}(T_xX)$ of dimension $d$ for every general point $x \in X$. Therefore one can apply the main theorem of \cite{MR2232023} (= Theorem \ref{T:AraMain}) to conclude that the $H$-rationally connected quotient $\pi^\circ: X^\circ \longrightarrow Y^\circ$ admits the structure of a projective space bundle. But since the Picard number of $X$ is 1, $Y^\circ$ is a point by Proposition \ref{P:trivMRC}. Therefore $X \simeq \mathbb{P}^n$ as desired, and this proves Theorem \ref{T:main}. $\Box$

\bibliographystyle{skalpha}
\bibliography{notes2b}

\end{document}